\documentclass[10pt]{amsart}
\usepackage{amssymb}
\usepackage{parskip}
\title{M\"obius Polynomials and Splitting Algebras of Direct Products of Posets}
\author{Susan Durst}
\date{\today}

\newcommand{\polynom}{p}
\renewcommand{\phi}{\varphi}

\newtheorem{theorem}{Theorem}[section]
\newtheorem{prop}{Proposition}[section]
\newtheorem{lemma}{Lemma}[section]
\newtheorem{cor}{Corollary}[section]

\theoremstyle{definition}
\newtheorem{defn}{Definition}[section]

\begin{document}

\maketitle

\begin{abstract}
In this paper, we will study the M\"obius polynomial, an invariant of ranked posets that arises in the study of splitting algebras.  We will present a formula for the M\"obius polynomial of the direct product of posets in terms of the M\"obius polynomials of the factors.  We will then use this formula to calculate Hilbert series and graded trace generating functions associated to the splitting algebras of the Boolean algebra and the poset of factors of a natural number $n$.
\end{abstract}

\section{Introduction}

Given a ranked poset $P$ with rank function $\rho$ and unique minimal element
$*$, we define the \textbf{M\"obius polynomial} of the poset $P$ to be
\[\mathcal{M}_P(z)=\sum_{p\leq q\in P}\mu(p,q)z^{\rho(q)-\rho(p)},\]
where $\mu$ is the M\"obius function on the poset $P$.

The M\"obius polynomial was introduced by Retakh, Serconek, and Wilson in 2008, in their study of $A(P)$, the splitting algebra of a poset $P$. \cite{HilbertSeriesPaper}.  In it, they show that the Hilbert series of $A(P)$ is given by
\[H(A(P),z)=\frac{1-z}{1-z\mathcal{M}_P(z)}.\]
The M\"obius polynomial also appears in a 2009 paper by Duffy \cite{ColleensThesis}. In this paper, she studies the
graded trace generating function of an automorphism $\sigma$ of $A(P)$
induced by an automorphism of the poset $P$, also called $\sigma$. She shows
that the graded trace generating function is given by
\[Tr_\sigma(A(\Gamma),z)=\frac{1-z}{1-z\mathcal{M}_{P_\sigma}(z)},\]
where $P_\sigma$ is the subposet of $P$ consisting of the fixed points of $P$
under the automorphism $\sigma$.

In this paper we will use both of these formulas to explore the behavior of the
Hilbert series and trace generating functions of direct products of posets. The
structure of the paper is as follows. In Section~\ref{Definitions}, we will present basic
definitions. In Section~\ref{PolynomialIncidence}, we will define a polynomial generalization of the incidence
algebra.  In Section~\ref{DirectProducts}, we will use this generalization to prove that given posets $P$ and $Q$, we
have
\[\mathcal{M}_{P\times Q}(z)=\mathcal{M}_P(z)\mathcal{M}_Q(z).\]
In Section~\ref{HilbertSeries}, we will use this result to derive a formula for calculating the Hilbert series of a direct product of posets.  We will use this to present a new proof of Theorem~3 from~\cite{HilbertSeriesPaper}, which states that the Hilbert series of the Boolean algebra $2^n$ is given by
\[H(A(2^n,z))=\frac{1-z}{1-z(2-z)^n}.\]
In Section~\ref{TraceFunctions}, we will present a formula for claculating the graded trace generating function of an automorphism $\sigma$ of the algebra $A(P)$ arising from permuting isomorphic copies of $Q_k$ in a poset of the form $P=\prod_{i=1}^n Q_i$.  We will use this in a new proof of Theorem 6.1 from \cite{ColleensThesis}, which states that given an automorphism $\sigma$ of $2^n$ permuting the $n$ atomic elements, with cycle decomposition $\sigma_1\ldots\sigma_r$, with each cycle $\sigma_j$ of length $i_j$, we have
\[Tr_\sigma(A(2^n),z)=\frac{1-z}{1-z\prod_{j=1}^r(2-z^{i_j})}.\]

In Section~\ref{FactorsOfN}, we will explore the poset $P_n$ of factors of $n$. We will
calculate the Hilbert series of $A(P_n)$, and find the graded trace generating
function $Tr_\sigma(A(P_n),z)$ in the case where $\sigma$ is the automorphism
of $A(P_n)$ induced by permuting elements of $P_n$ that correspond to the same
prime factor of $n$.

\section{Basic Definitions And Results}
\label{Definitions}

\begin{defn}
Let $(P,\leq)$ be a poset. Given $p,q\in P$, we write $p\lessdot q$, and say
that $p$ is \textbf{covered by} $q$ if for any $r\in P$ with $p\leq r\leq q$,
we have $r=p$ or $r=q$.
\end{defn}

\begin{defn} A \textbf{ranked poset} is a poset $(P,\leq)$, together with a rank function
$\rho:P\rightarrow\mathbb{N}$ satisfying
\begin{itemize}
\item[(i)] For an element $p\in P$, $\rho(p)=0$ if and only if $p$ is minimal.
\item[(ii)] Whenever $p,q\in P$ such that $p\lessdot q,$ we have
$\rho(q)=\rho(p)+1$.
\end{itemize}
\end{defn}

\begin{defn} A \textbf{generalized ranked poset} is a poset $(P,\leq)$, together with a
generalized rank function $|\cdot|:P\rightarrow\mathbb{N}$ satisfying
\begin{itemize}
\item[(i)] For an element $p\in P$, $|p|=0$ if and only if $p$ is minimal.
\item[(ii)] Whenever $p<q$ in $P$, we have $|p|<|q|$.
\end{itemize}
\end{defn}

Notice that a rank function is always a generalized rank function, and that a
ranked poset is always a generalized ranked poset.

\begin{defn}
Given a locally finite poset $P$, and elements $p,q\in P$,the \textbf{closed
interval} $[p,q]\subseteq P$ is defined by
\[[p,q]=\{r\in P : p\leq r\leq q\}.\]
If $p\not\leq q$, then $[p,q]=\emptyset$. Let $I_P$ be the collection of closed
intervals in $P$. The \textbf{M\"obius function} on the poset $P$ is the
function $\mu:I_P\rightarrow \mathbb{Z}$ given by
\begin{itemize}
\item[(i)]$\mu([p,p])=1$.
\item[(ii)] $\displaystyle{\mu([p,q])=-\sum_{p\leq r< q}\mu(p,q)}$.
\end{itemize}
If $p\not\leq q$, then we define $\mu([p,q])=0$. For ease of notation, we will
write $\mu(p,q)$ rather than $\mu([p,q])$.
\end{defn}

\begin{defn}
Given a finite generalized ranked poset $P$, we define the \textbf{M\"obius
polynomial} associated to the poset $P$ by
\[\mathcal{M}_P(z)=\sum_{p\leq r<q}\mu(p,q)z^{|q|-|p|}.\]
\end{defn}

\begin{defn}
Given a finite ranked poset $P$, we can construct a directed graph $(P,E_P)$
called the \textbf{Hasse diagram} of $P$. In this graph the edge set is given
by
\[E_P=\{(q,p) : p\lessdot q\}.\]
An ordered $n$-tuple of edges $\pi=(e_1,e_2,\ldots,e_n)$ is called a
\textbf{path} if $h(e_i)=t(e_{i+1})$ for $1\leq i<n$. Given a path
$\pi=(e_1,e_2,\ldots,e_n)$, we define $h(\pi)=h(e_n)$, and $t(\pi)=t(e_1)$. We
define an equivalence relation $\approx$ on the collection of paths as follows:
Given two paths $\pi_1$ and $\pi_2$, we say $\pi_1\approx\pi_2$ if and only if
$h(\pi_1)=h(\pi_2)$ and $t(\pi_1)=t(\pi_2)$.
\end{defn}

\begin{defn}
Let $\mathbb{F}$ be a field. Following \cite{AlgebrasAssocToDirGraphs}, we define an $\mathbb{F}$-algebra
$A(P)$ as follows. Let $\mathbb{F}\langle E_P\rangle$ be the free algebra
generated by the edges of the Hasse diagram of $P$. For every path
$\pi=(e_1,\ldots,e_n)$ in the Hasse diagram, we associate a polynomial
$\polynom_\pi\in \mathbb{F}\langle E_P
\rangle[t]$, given by
\[p_\pi(t)=(t-e_1)(t-e_2)\ldots(t-e_n)=\sum_{i=0}^\infty e(\pi,i)t^i\]
The algebra $A(P)$ is defined to be $\mathbb{F}\langle E_P\rangle/R$, where $R$
is the ideal generated by
\[\{e(\pi_1,i)-e(\pi_2,i) : \pi_1\approx\pi_2\}.\]
Since $R$ is a homogeneous ideal, the algebra $A(P)$ inherits the natural
grading from $\mathbb{F}\langle E_P\rangle$.
\end{defn}

\begin{defn}
Given a graded algebra $A=\bigoplus A_i$, the \textbf{Hilbert series} of $A$ is
given by
\[H(A,z)=\sum_{i=0}^\infty \dim(A_i)z^i.\]
\end{defn}

In \cite{HilbertSeriesPaper}, Retakh, Serconek, and Wilson prove the following result: 
\begin{theorem} 
\label{HilbertSeriesResult}
If $A$ is a finite ranked poset with splitting algebra $A(P)$, then the Hilbert series of $A(P)$ is given by
\[H(A(P),z)=\frac{1-z}{1-z\mathcal{M}_P(z)}.\]
\end{theorem}

\begin{defn}
Let $A=\bigoplus A_i$ be a graded algebra, and let $\sigma:A\rightarrow A$ be a
graded algebra automorphism. Then $\sigma$ acts on each graded piece of $A$ as
a vector space automorphism. The \textbf{graded trace generating function} of
$\sigma$ acting on $A$ is defined to be
\[Tr_\sigma(A,z)=\sum_{i=0}^\infty Tr\left(\sigma|_{A_i}\right)z^i.\]
\end{defn}

\begin{defn}
Let $(P,\leq)$ be a finite ranked poset, and let $\sigma$ be an automorphism of
$P$. We define $(P^\sigma,\leq)$ to be the generalized ranked poset consisting
of the elements of $P$ that are fixed by the automorphism $\sigma$, with
ordering relation inherited from the poset $P$.
\end{defn}

An automorphism $\sigma$ of a poset $P$ induces an automorphism of the Hasse diagram
$(P,E_P)$, and thus of the algebra $A(P)$. We will also call this automorphism
$\sigma$. In \cite{ColleensThesis}, Duffy proves the following result:

\begin{theorem}
\label{GradedTraceResult}
If $\sigma$ is an automorphism of a finite ranked poset $P$, then the graded trace generating function of $\sigma$ acting on $A(P)$ is given
by
\[Tr_{\sigma}(A(P),z)=\frac{1-z}{1-z\mathcal{M}_{P^\sigma}(z)}.\]
\end{theorem}

In this paper, we will use Theorems~\ref{HilbertSeriesResult} and \ref{GradedTraceResult}, together with some original
results regarding the M\"obius polynomial, in order to calculate Hilbert series
and graded trace generating functions related to direct products of posets.

\section{The Polynomial Incidence Algebra}
\label{PolynomialIncidence}

We begin by recalling a definition of the \textbf{incidence algebra} of a
locally finite poset $(P,\leq)$. Let $V_P$ be the vector space over a field
$\mathbb{F}$ with basis $P$. The incidence algebra $I(P)$ is the collection of
linear maps $\varphi:V_P\otimes V_P\rightarrow\mathbb{F}$ satisfying
\[\varphi(p\otimes p^\prime)=0\hspace{8pt}\text{if}\hspace{8pt}p\not\leq
p^\prime,\]
with the algebra operation $\cdot$ given by convolution. That is, for $\varphi$
and $\psi$ in $I(P)$, the map $(\varphi\cdot\psi)$ is given by
\[(\varphi\cdot\psi)(p\otimes q)=\sum_{p\leq r\leq q}\varphi(p\otimes
r)\psi(r\otimes q).\]

The identity of the incidence algebra is the function
\[\delta(p\otimes q)=\left\{
\begin{array}{cl}
0&\text{if }p\neq q\\
1&\text{if }p= q\\
\end{array}\right..\]  
We also define an element $\zeta$ of the incidence algebra, called the
\textbf{zeta function}, given by
\[\zeta(p\otimes q)=\left\{
\begin{array}{cl}
0&\text{if }p\not\leq q\\
1&\text{if }p\leq q\\
\end{array}\right..\]  
The element $\zeta$ has a multiplicative inverse in $I(P)$, given by
$\zeta^{-1}(p \otimes q)=\mu(p,q)$. This is an alternate definition of the
M\"obius function.

Given a generalized ranked poset $P$, we define the \textbf{polynomial
incidence algebra} $I_z(P)$ of $(P,\leq)$ to be the collection of linear maps
$\varphi: V_P\otimes V_P\rightarrow \mathbb{F}[z]$ satisfying
\[\begin{array}{lcc}
\varphi(p\otimes p^\prime)=0&\text{if}&p\not\leq p^\prime,\\
\varphi(p\otimes p^\prime)\in\mathbb{F} z^{|p^\prime|-|p|}&\text{if}&p\leq
p^\prime.\\
\end{array}\]
Again, the operation $\cdot$ will be given by convolution. Notice that if we
substitute 1 for $z$, we will obtain the usual incidence algebra $I(P)$. The
identity of the polynomial incidence algebra is the function $\delta_z,$ given
by
\[\delta_z(p\otimes q)=\left\{
\begin{array}{cl}
0&\text{if }p\neq q\\
1&\text{if }p= q\\
\end{array}\right..\]  
We define $\zeta_z$, the \textbf{polynomial zeta function} by 
\[\zeta_z(p\otimes q)=\left\{
\begin{array}{cl}
0&\text{if }p\not\leq q\\
z^{|q|-|p|}&\text{if }p\leq q\\
\end{array}\right..\]
Notice that $\zeta_z(p\otimes q)=\zeta(p\otimes q)z^{|q|-|p|},$ so if we set
$z=1$, we recover the usual zeta function. The inverse of the polynomial zeta
function is the \textbf{polynomial M\"obius function} $\mu_z$, given by
\[\mu_z(p\otimes q)=\left\{
\begin{array}{cl}
0&\text{if }p\not\leq q\\
\mu(p,q)z^{|q|-|p|}&\text{if }p\leq q\\
\end{array}\right..\]
This is easy to check:
\begin{eqnarray*}
(\zeta_z\cdot\mu_z)(p\otimes q)&=&\sum_{p\leq r\leq q}\zeta_z(p\otimes
r)\mu_z(r\otimes q)\\
&=&\sum_{p\leq r\leq q}\left(\zeta(p\otimes r)z^{|r|-|p|}\right)\left(\mu(r,
q)z^{|q|-|r|}\right)\\
&=&\left(\sum_{p\leq r\leq q}\zeta(p\otimes r)\zeta^{-1}(r\otimes
q)\right)z^{|q|-|p|}
\end{eqnarray*}
Since $\zeta$ and $\zeta^{-1}$ are inverses in $I(P)$, we know that this will
be equal to 1 if $p=q$, and to 0 otherwise. Thus $\zeta_z\cdot\mu_z=\delta_z$.
Similarly, we can show that $\mu_z\cdot\zeta_z=\delta_z$.

If $P$ is a finite generalized ranked poset, then we can find a total ordering of $P$ that respects its partial ordering.  That is, we can write $P=\left\{p_1,p_2,\ldots, p_{|P|}\right\}$ such that $p_i\leq p_j$ implies $i\leq j$.  We know that any linear map $\phi:V_P\otimes V_P\rightarrow \mathbb{F}[z]$ is uniquely defined by its values on $p\otimes p^\prime$ for $p,p^\prime\in \mathbb{F}$, so there exists a unique $|P|$ by $|P|$ matrix $M(\phi)=[a_{ij}]$ with entries in $F[z]$ such that $a_{ij}=\phi(p_i\otimes p_j)$.  The convolution operation corresponds to matrix multiplication in this context, giving us
\[M(\phi\cdot\psi)=M(\phi)M(\psi).\]
Thus $M$ is an injective algebra homomorphism whose image consists of all matrices with entries in $\mathbb{F}[z]$ such that $a_{ij}=0$ whenever $p_i\not\leq p_j$.  Furthermore, for any $v,w\in V_P$, we have $\phi(v\otimes w)=v^TM(\phi)w.$

\section{The M\"obius Polynomial Of Direct Products Of Posets}
\label{DirectProducts}

Let $P$ be a finite generalized ranked poset, and let $V_P$ be the vecor space
over $\mathbb{F}$ with basis $P$. Define the element $\textbf{1}\in V_P$ by
$\textbf{1}=\sum_{p\in P}p$. Then $\textbf{1}\otimes\textbf{1}$ is equal to
$\sum_{p,q\in P}p\otimes q,$ and so in the polynomial incidence algebra
$I_z(P)$, we have
\[\mu_z(\textbf{1}\otimes\textbf{1})=\sum_{p,q\in P}\mu_z(p\otimes
q)=\sum_{p\leq q}\mu(p,q)z^{|q|-|p|}=\mathcal{M}_P(z).\]

\begin{defn}
Given generalized ranked posets $(P,\leq_P)$ and $(Q,\leq_Q)$ with
corresponding generalized rank functions $|\cdot|_P$ and $|\cdot|_Q$, we define
the \textbf{direct product} of the two posets to be $(P\times Q,\leq_{P\times
Q})$, with order relation given by $(p,q)\leq_{P\times Q}(p^\prime,q^\prime)$
if and only if $p\leq_P p^\prime$ and $q\leq_Q q^\prime$, and rank function
given by $|\cdot|_{P\times Q}((p,q))=|p|_P+|q|_Q$.
\end{defn}

\begin{theorem}
\label{isomorphism}
If $(P,\leq_P)$ and $(Q,\leq_Q)$ are finite generalized ranked posets, then 
\[I_z(P\times Q)\cong I_z(P)\otimes I_z(Q).\]
\end{theorem}

\begin{proof}
For ease of notation, we will use $|\cdot|$ as the rank function for $P\times
Q$, $|\cdot|_P$ for the rank function of $P$, and $|\cdot|_Q$ for the rank
function of $Q$.

We know that $I_z(P\times Q)$ is the collection of linear maps $\varphi$ from
$V_{P\times Q}\otimes V_{P\times Q}$ to $\mathbb{F}[z]$ such that
\begin{itemize}
\item[(i)]$\varphi((p,q)\otimes (p^\prime,q^\prime))=0$ if
$(p,q)\not\leq(p^\prime,q^\prime)$.
\item[(ii)] $\displaystyle{\varphi((p,q)\otimes (p^\prime,q^\prime))\in
\mathbb{F} z^{|(p^\prime,q^\prime)|-|(p,q)|}}$ if $(p,q)\leq
(p^\prime,q^\prime)$.
\end{itemize}
The algebra operation is given by convolution.

Consider the map
\[B:(I_z(P)\times I_z(Q))\rightarrow I_z(P\times Q)\]
given by
\[B(\phi_P, \phi_Q)((p,q)\otimes (p^\prime,q^\prime))=\phi_P(p\otimes p^\prime)\phi_Q(q\otimes q^\prime).\]
We can show that $B$ is a balanced product map.  We have
\[B((\phi_P+\phi_P^\prime),\phi_Q)((p,q)\otimes(p^\prime,q^\prime))=(\phi_P(p\otimes p^\prime)
+\phi_P^\prime(p\otimes p^\prime))\phi_Q(q\otimes q^\prime).\]
The right side of the expression expands out to
\[(\phi_P(p\otimes p^\prime)\phi_Q(q\otimes q^\prime))+(\phi_P^\prime(p\otimes p^\prime)\phi_Q(q\otimes q^\prime)),\]
which is equal to
\[B(\phi_P,\phi_Q)((p,q)\otimes(p^\prime,q^\prime))+B(\phi_P^\prime,\phi_Q)((p,q)\otimes(p^\prime,q^\prime)).\]
It follows that $B(\phi_P+\phi_P^\prime,\phi_Q)=B(\phi_P,\phi_Q)+B(\phi_P^\prime,\phi_Q).$  Similarly, we can show that $B(\phi_P,\phi_Q+\phi_Q^\prime)=B(\phi_P,\phi_Q)+B(\phi_P,\phi_Q^\prime)$.

We can also see that for a constant $a\in \mathbb{F}$ we have
\begin{eqnarray*}
B(a \phi_P,\phi_Q)((p,q)\otimes(p^\prime,q^\prime))&=& (a\phi_P(p\otimes p^\prime))(\phi_Q(q\otimes q^\prime))\\
&=& (\phi_P(p\otimes p^\prime))(a\phi_Q(q\otimes q^\prime))\\
&=& B(\phi_P,a \phi_Q)((p,q)\otimes(p^\prime,q^\prime)).
\end{eqnarray*}

Since $B$ is a balanced product from $I_z(P)\times I_z(Q)$ to $I_z(P\times Q)$, the universal property of the tensor product guarantees a unique algebra homomorphism
\[\Phi:I_z(P)\otimes I_z(Q)\rightarrow I_z(P\times Q)\]
such that $\Phi(\phi_P\otimes \phi_Q)=B(\phi_P,\phi_Q)$ for all $\phi_P\in I_z(P)$ and $\phi_Q\in I_z(Q)$.  We wish to show that $\Phi$ is an isomorphism.  We will do so by constructing an inverse function $\Psi:I_z(P\times Q)\rightarrow I_z(P)\otimes I_z(Q)$.  

We begin by defining a function $\phi_{(p,p^\prime)}\in I_z(P)$ for each $p,p^\prime\in P$ and a function $\phi_{(q,q^\prime)}\in I_z(Q)$ for each $q,q^\prime\in I_z(Q)$, as follows:

 If $p\not\leq p^\prime$, then we set $\phi_{(p,p^\prime)}\equiv 0$.  If $p\leq p^\prime$, we set
\[\phi_{(p,p^\prime)}(r\otimes r^\prime)=\left\{
\begin{array}{cl}
z^{\left(|r^\prime|_P-|r|_P\right)}&\text{if }r=p, r^\prime=p^\prime\\
0&\text{otherwise}
\end{array}\right..\]

Similarly, if $q\not\leq q^\prime$, we set $\phi_{(q,q^\prime)}\equiv 0$.  If $q\leq q^\prime$, we set
\[\phi_{(q,q^\prime)}(s\otimes s^\prime)=\left\{
\begin{array}{cl}
z^{\left(|s^\prime|_P-|s|_P\right)}&\text{if }s=q, s^\prime=q^\prime\\
0&\text{otherwise}
\end{array}\right..\]

For any $p,p^\prime\in P$ and $q,q^\prime\in Q$, we have
\[\Phi\left(\phi_{(p,p^\prime)}\otimes \phi_{(q,q^\prime)}\right)((r,s)\otimes (r^\prime,s^\prime))=\left(\phi_{(p,p^\prime)}(r\otimes r^\prime)\right)\left(\phi_{(q,q^\prime)}(s\otimes s^\prime)\right).\]

If $(p,q)\not\leq (p^\prime,q^\prime)$, then $\Phi\left(\phi_{(p,p^\prime)}\otimes \phi_{(q,q^\prime)}\right)\equiv 0$.  Otherwise, it is the function which takes on the value $z^{(|(p^\prime,q^\prime)|-|(p,q)|)}$ at $(p,q)\otimes(p^\prime,q^\prime)$, and $0$ everywhere else.

Note that any element of $I_z(P\times Q)$ can be determed by its values on $(p,q)\otimes (p^\prime,q^\prime)$ for $p,p^\prime\in P$ and $q,q^\prime\in Q$ with $(p,q)\leq (p^\prime,q^\prime)$.  Let $\phi\in I_z(P\times Q)$ be defined by
\[\phi((p,q)\otimes(p^\prime,q^\prime))=a(p,q,p^\prime,q^\prime) z^{(|(p^\prime,q^\prime)|-|(p,q)|)}.\]

We define
\[\Psi(\phi)=\sum_{(p,q)\leq(p^\prime,q^\prime)}a(p,q,p^\prime,q^\prime)\left(\phi_{(p,p^\prime)}
\otimes\phi_{(q,q^\prime)}\right).\]
It follows that 
\[\Phi\circ\Psi(\phi)=\sum_{(p,q)\leq(p^\prime,q^\prime)}a(p,q,p^\prime,q^\prime) \Phi\left(\phi_{(p,p^\prime)}
\otimes\phi_{(q,q^\prime)}\right).\]
The value of this function at $(p,q)\otimes (p^\prime,q^\prime)$ is
\begin{eqnarray*}
a(p,q,p^\prime,q^\prime)\Phi\left(\phi_{(p,p^\prime)}\otimes \phi_{(q,q^\prime)}\right)&=&a(p,q,p^\prime,q^\prime)z^{(|(p^\prime,q^\prime)|-|(p,q)|)}\\
&=&\phi((p,q)\otimes(p^\prime,q^\prime))
\end{eqnarray*}
if $(p,q)\leq(p^\prime,q^\prime)$, and 0 otherwise.  Thus the function is equal to $\phi$, and so $\Phi\circ\Psi$ is the identity on $I_z(P\times Q)$.

Now we notice that any element of $I_z(P)$ is entirely determined by its values on $(p,p^\prime)$ for $p\leq p^\prime$.  This means that each element $\phi_P\in I_z(P)$ can be written as
\[\sum_{p\leq p^\prime}b(p,p^\prime)\phi_{(p,p^\prime)}.\]
Similarly, element $\phi_Q\in I_z(Q)$ can be written as
\[\sum_{q\leq q^\prime}c(q,q^\prime)\phi_{(q,q^\prime)},\]
and any elment of $I_z(P)\otimes I_z(Q)$ can be written as
\[\sum_{\substack{p\leq p^\prime\\q\leq q^\prime}}(b(p,p^\prime)\phi_{(p,p^\prime)})\otimes(c(q,q^\prime)
\phi_{(q,q^\prime)}).\]
If we set $a(p,q,p^\prime,q^\prime)=b(p,p^\prime)c(q,q^\prime)$, then this expression becomes
\[\sum_{\substack{p\leq p^\prime\\q\leq q^\prime}}a(p,q,p^\prime,q^\prime)\left(\phi_{(p,p^\prime)}\otimes \phi_{(q,q^\prime)}\right).\]
Applying $\Phi$ to this expression, we obtain
\[\sum_{\substack{p\leq p^\prime\\q\leq q^\prime}}a(p,q,p^\prime,q^\prime)\Phi\left(\phi_{(p,p^\prime)}\otimes \phi_{(q,q^\prime)}\right)\in I_z(P\times Q),\]
the function which takes each $(p,q)\otimes(p^\prime,q^\prime)$to $a(p,q,p^\prime,q^\prime)z^{(|(p^\prime,q^\prime)|-|(p,q)|)}$ if $(p,q)\leq (p^\prime,q^\prime)$, and to zero otherwise.

Applying $\Psi$ to this expression gives us 
\[\sum_{(p,q)\leq(p^\prime,q^\prime)} a(p,q,p^\prime,q^\prime)\left(\phi_{(p,p^\prime)}\otimes\phi_{(q,q^\prime)}\right)=\sum_{\substack{p\leq p^\prime\\q\leq q^\prime}}(b(p,p^\prime)\phi_{(p,p^\prime)})\otimes(c(q,q^\prime)
\phi_{(q,q^\prime)}).\]
It follows that $\Psi\circ\Phi$ is the identity on $I_z(P)\otimes I_z(Q)$, and thus our proof is complete.
\end{proof}

\begin{theorem}
\label{DirectProductResult}
Let $P$ and $Q$ be finite generalized ranked posets.  Then
\[\mathcal{M}_{P\times Q}(z)=\mathcal{M}_P(z)\mathcal{M}_Q(z).\]
\end{theorem}

\begin{proof}

We know that $\zeta^P_z(p\otimes p^\prime)$ is $z^{(|p^\prime|-|p|)}$ if $p\leq p^\prime$, and zero otherwise.  Similarly, $\zeta_z^Q(q\otimes q^\prime)$ is $z^{(|q^\prime|-|q|)}$ if $q\leq q^\prime$ and zero otherwise.  It follows that, in the notation of the proof of Theorem~\ref{isomorphism},
\[\Phi\left(\zeta^P_z\otimes \zeta_z^Q\right)\left((p,q)\otimes(p^\prime,q^\prime)\right)=\left(\zeta^P_z(p\otimes p^\prime)\right)\left(\zeta_z^Q(q\otimes q^\prime)\right)\]
is equal to $z^{(|(p^\prime,q^\prime)|-|(p,q)|)}$ if $(p,q)\leq (p^\prime,q^\prime)$, and zero otherwise.  It follows that
\[\Phi\left(\zeta_z^P\otimes\zeta_z^Q\right)=\zeta_z^{P\times Q}\]
Since $\mu_z^P\otimes \mu_z^Q$ is the inverse of $\zeta_z^P\otimes\zeta_z^Q$ in $I_z(P)\otimes I_z(Q)$, it follows that
\[\Phi\left(\mu_z^P\otimes\mu_z^Q\right)=\mu_z^{P\times Q}.\]
Thus,
\begin{eqnarray*}
\mathcal{M}_{P\times Q}(z)&=&\mu_z^{P\times Q}(\textbf{1}_{P\times Q}\otimes\textbf{1}_{P\times Q})\\
&=& \hspace{-18pt}\sum_{(p,q),(p^\prime,q^\prime)\in P\times Q}\hspace{-18pt}\mu_z^{P\times Q}((p,q)\otimes(p^\prime,q^\prime)\\
&=& \hspace{-18pt}\sum_{(p,q),(p^\prime,q^\prime)\in P\times Q}\hspace{-18pt}\mu_z^P(p\otimes p^\prime)\mu_z^Q(q\otimes q^\prime)\\
&=&\left(\sum_{p,p^\prime\in P}\mu_z^P(p\otimes p^\prime)\right)\left(\sum_{q,q^\prime\in Q}\mu_z^Q(q\otimes q^\prime\right)\\
&=&\mu_z^P(\textbf{1}_P\otimes\textbf{1}_P)\mu_z^Q(\textbf{1}_Q \otimes\textbf{1}_Q)\\
&=&\mathcal{M}_P(z)\mathcal{M}_Q(z).
\end{eqnarray*}
\end{proof}

\section{Hilbert Series of $A(P\times Q)$}
\label{HilbertSeries}

Combining Theorem~\ref{HilbertSeriesResult} and Theorem~\ref{DirectProductResult}, we immediately obtain the following theorem regarding the Hilbert series of the direct product of posets:

\begin{prop}
Let $P$ and $Q$ be finite ranked posets.  Then
\[H(A(P\times Q),z)=\frac{1-z}{1-z\mathcal{M}_P(z)\mathcal{M}_Q(z)}.\]
\end{prop}

We can use this to provide a simple proof of Theorem 3 as presented in \cite{HilbertSeriesPaper}:

\begin{theorem}
\label{Boolean}
Let $2^{[n]}$ be the lattice of subsets of the set $[n]=\{1,\ldots, n\}$.  Then
\[H\left(2^{[n]},z\right)=\frac{1-z}{1-z(2-z)^n}\]
\end{theorem}

\begin{proof}
Consider the poset $\{P,\leq\}$ where $P=\{x,y\}$, and $x<y$. The poset $2^{[n]}$ is isomorphic to a direct product of $n$ copies of the poset $P$.  We have
\[M\left(\zeta_z^P\right)=\left[\begin{array}{cc}
1&z\\
0&1
\end{array}\right],\]
and thus
\[M\left(\mu_z^P)\right)=\left[\begin{array}{cc}
1&z\\
0&1
\end{array}\right]^{-1}=\left[\begin{array}{cc}
1&-z\\
0&1
\end{array}\right]\]
Thus
\[\mathcal{M}_P(z)=\mu_z^P(\textbf{1}\otimes\textbf{1})=\left[
\begin{array}{cc}
1&1
\end{array}
\right]\left[
\begin{array}{cc}
1&-z\\
0&1
\end{array}\right]\left[\begin{array}{c}
1\\
1
\end{array}\right]
=2-z\]

It follows that by Theorem~\ref{DirectProductResult}, 
\[\mathcal{M}_{2^{[n]}}(z)=\mathcal{M}_{P^n}(z)=(2-z)^n,\]
and so
\[H\left(2^{[n]},z\right)=\frac{1-z}{1-z(2-z)^n}\]
\end{proof}

\section{Graded Trace Generating Functions And Direct Products}
\label{TraceFunctions}

Suppose $P$ is the direct product of $n$ posets $Q_1,\ldots,Q_n$, some of which are isomorphic.  Then there exists a subgroup of $S_n$ which acts on $P$, and thus $A(P)$, by permuting copies of $Q_k$ which are isomorphic to each other.  If $\sigma$ is an element of this subgroup, we call $\sigma$ a \textbf{factor-shuffling} automorphism of $P$.

Our goal in this section is to prove the following theorem:

\begin{theorem}
\label{shuffling}
Consider a poset $P=\prod_{k=1}^n Q_k$.  Let $\sigma$ be a factor-shuffling automorphism of $P$, with cycle decomposition in $S_n$ given by $\sigma_1\sigma_2\ldots\sigma_r$, with each $\sigma_j$ acting on factors isomorphic to $Q_j$, and in which the cycle length of each $\sigma_j$ is given by $i_j$.  Then
\[\mathcal{M}_{P^{\sigma}}(z)=\prod_{j=1}^r\mathcal{M}_{Q_j}\left(z^{i_j}\right).\]
\end{theorem}

This together with Theorem~\ref{GradedTraceResult} proved by Duffy in \cite{ColleensThesis} immediately gives us

\begin{cor}
If $P=\prod_{k=1}^n Q_k$, and $\sigma$ is a factor-shuffling automorphism of $P$ satisfying all the hypotheses of Theorem~\ref{shuffling}, then
\[Tr_\sigma(A(P),z)=\frac{1-z}{1-z\prod_{j=1}^r\mathcal{M}_{Q_j}\left(z^{i_j}\right)}.\]
\end{cor}

In our proof of these results, we will find the following definition to be useful:

\begin{defn}
Let $P$ be a generalized ranked poset with rank function $|\cdot|_P$.  We define $^{\times n}P$ to be the generalized ranked poset with the same underlying poset $P$, but with rank function given by \[|p|_{(^{\times n}P)}=n\cdot|p|_P.\]
\end{defn}

\begin{lemma}
\label{Cycle}
Let $Q$ be a generalized ranked poset, and let $P=Q^n$.  Let $\sigma$ be an automorphism of $P$ which cyclically permutes the copies of $Q$ in $P$.  Then \[P^\sigma\cong (^{\times n}Q).\]
\end{lemma}

\begin{proof}
An element $(q_1,\ldots,q_n)\in P$ is fixed by $\sigma$ if and only if
\[q_1=q_2=\ldots=q_n.\]
It follows that $P^\sigma$ consists of elements of the form $(q,\ldots,q)$, for $q\in Q$.  We have $(q,\ldots,q)\leq (q^\prime,\ldots,q^\prime)$ if and only if $q\leq q^\prime$.  For any $(q,\ldots,q)\in P$, we have
\[|(q,\ldots,q)|_P=n\cdot|q|_Q=|q|_{(^{\times n}Q)},\]
which completes our proof.
\end{proof}

\begin{lemma}
\label{ProductOfAut}
Let $P=Q\times R$, and let $\sigma:P\rightarrow P$ be an automorphism such that $\sigma((q,r))=(\phi(q),\psi(r))$ for automorphisms $\phi: Q\rightarrow Q$ and $\psi:R\rightarrow R$.  Then
\[P_\sigma=Q_\phi\times R_\psi.\]
\end{lemma}

\begin{proof}
If $(q,r)\in P_\sigma$, then $(q,r)$ is fixed by $\sigma$.  This is true if and only if $q$ is fixed by $\phi$ and $r$ is fixed by $\psi$.  Thus $(q,r)\in P_\sigma$ if and only if it is in $Q_\phi\times R_\psi$.
\end{proof}

\begin{lemma}
\label{MobOfnQ}
\[\mathcal{M}_{\left(^{\times n}Q\right)}(z)=\mathcal{M}_Q(z^n)\]
\end{lemma}

\begin{proof}
Both $I_z(^{\times n}Q)$ and $I_{z^n}(Q)$ consist of exactly the functions
\[\phi: Q\otimes Q\rightarrow \mathbb{F}[z]\]
satisfying the following two conditions:
\begin{itemize}
\item[(i)] $\phi(q,q^\prime)=0$ if $q\not\leq q^\prime$.
\item[(ii)] $\phi(q,q^\prime)\in z^{\left(n\cdot|q^\prime|_Q-n\cdot|q|_Q\right)}\mathbb{F}$ if $q\leq q^\prime$.
\end{itemize}
It follows that $I_z(^{\times n} Q)=I_{z^n}Q$.  In particular, we have $\mu_{(^{\times n}Q)}^z=\mu_Q^{z^n}$. This gives us
\begin{eqnarray*}
\mathcal{M}_{\left(^{\times n}Q\right)}(z)&=&\mu_{\left(^{\times n}Q\right)}^z(\textbf{1}\otimes \textbf{1})\\
&=&\mu_{Q}^{z^{n}}(\textbf{1}\otimes \textbf{1})\\
&=&\mathcal{M}_Q(z^n).
\end{eqnarray*}
\end{proof}

\begin{proof}[Proof of Theorem~\ref{shuffling}]
Let $P=\prod_{k=1}^n Q_k$, and let $\sigma$ be a factor-shuffling automorphism of $P$ with cycle decomposition in $S_n$ given by $\sigma_1\sigma_2\ldots\sigma_r$, with each $\sigma_j$ acting on factors isomorphic to $Q_j$, and in which the cycle length of each $\sigma_j$ given by $i_j$.  Then we have
\[P=(Q_1)^{i_1}\times (Q_2)^{i_2}\times\ldots\times(Q_r)^{i_r}.\]
By Lemmas~\ref{Cycle} and \ref{ProductOfAut}, this gives us
\[P_\sigma\cong\left(^{\times i_1}Q_1\right)\times\left(^{\times i_2}Q_2\right)\times\ldots\times\left(^{\times i_r}Q_r\right).\]
Thus we have
\[\mathcal{M}_P(z)=\prod_{j=1}^r\mathcal{M}_{\left(^{\times i_j}Q_j\right)}(z),\]
and Lemma~\ref{MobOfnQ} gives us
\[\mathcal{M}_P(z)=\prod_{j=1}^r\mathcal{M}_{Q_j}\left(z^{i_j}\right)\]
\end{proof}

This, together with Theorem~\ref{GradedTraceResult} from \cite{ColleensDissertation} gives us a proof of Proposition 5.4 from \cite{ColleensDissertation}:

\begin{prop}
 Let $2^{[n]}$ be the lattice of subsets of the set $[n]=\{1,\ldots,n\}$, and let $\sigma$  be an automorphism of $P$ permuting the $n$ single-element sets, with cycle decomposition $\sigma_1\sigma_2\ldots\sigma_r$, with each cycle $\sigma_j$ of length $i_j$.  Then
 \[Tr_\sigma(2^{[n]},z)=\frac{1-z}{1-z\prod_{j=1}^r(2-z^{i_j})}\]
\end{prop}
\begin{proof}
As stated in the proof of Proposition \ref{Boolean}, $2^{[n]}\cong P^n$, where $P$ is the two-element poset $\{x,y\}$ with $x<y$.  We have $\mathcal{M}(P)=2-z$.  This together with Theorem~\ref{DirectProductResult} gives us our result.
\end{proof}

\section{The Poset Of Factors Of $n$}
\label{FactorsOfN}

Here we will use the results from Sections \ref{HilbertSeries} and \ref{TraceFunctions} to calculate the Hilbert series and graded trace generating functions associated to the poset $P$ of factors of a natural number $n$.  We can define a rank function $|\cdot|$ on $P$ by defining $|m|$  to be $\Omega(m)$, the number of prime factors of $m$ counting multiplicity. We will prove the following results:

\begin{theorem}
\label{HilbertSeriesFactors}
Let $P$ be the ranked poset of factors of $n$, ordered by divisibility.    If $n$  has prime factorization $p_1^{s_1}p_2^{s_2}\ldots p_k^{s_k}$, for distinct primes $p_1,\ldots,p_k$, then
\[H(A(P),z)=\frac{1-z}{1-z\prod_{j=1}^{r}\left(1+s_j(1-z)\right)}.\]
\end{theorem}

\begin{theorem}
\label{TraceFunctionFactors}
Let $P$ be the ranked poset of factors of $n$, ordered by divisibility.    Suppose that $n$  has prime factorization $p_1^{s_1}p_2^{s_2}\ldots p_k^{s_k}$, for distinct primes $p_1,\ldots,p_k$, and that $\sigma$ is a permutation which acts on $P$ by permuting the primes which appear with the same multiplicity in the factorization of $n$.  If $\sigma$ has cycle decomposition $\sigma_1\sigma_2\ldots\sigma_r$ with each cycle $\sigma_j$ of length $i_j$, permuting primes that appear with multiplicity $s_j$, then
\[Tr_\sigma(A(P),z)=\frac{1-z}{1-z\prod_{j=1}^r(1+s_j\left(1-z^{i_j})
\right)}\]
\end{theorem}

\begin{lemma}
Let $C_s$ be the chain of length $s$.  That is, $C_s=\{1,2,\ldots,s\}$, with the usual ordering.  Then
\[\mathcal{M}_{C_s}(z)=1+s(1-z).\]
\end{lemma}

\begin{proof}
We have
\[M\left(\zeta_z^{C_s}\right)=\left[
\begin{array}{ccccc}
1&t&z^2&\cdots&z^s\\
0&1&z&\cdots&z^{s-1}\\
0&0&1&\cdots&z^{s-2}\\
\vdots&\vdots&\vdots&\ddots&\vdots\\
0&0&0&\cdots&1\\
\end{array}
\right],\]
and so
\[M\left(\mu_z^{C_s}\right)=\left[
\begin{array}{ccccc}
1&-z&0&\cdots&0\\
0&1&-z&\cdots&0\\
0&0&1&\cdots&\vdots\\
\vdots&\vdots&\vdots&\ddots&-z\\
0&0&0&\cdots&1\\
\end{array}
\right].\]
It follows that 
\[\mathcal{M}_{C_s}(z)=\mu_z^{C_s}(\textbf{1}\otimes \textbf{1})=(s+1)-zs=1+s(1-z)\]
\end{proof}

\begin{lemma}
Let $P$ be the ranked poset of factors of $n$, ordered by divisibility.  Suppose that $n$ has prime factorization $p_1^{s_1}p_2^{s_2}\ldots p_k^{s_k}$ for distinct primes $p_1,\ldots,p_k$.  Then
\[P\cong C_{s_1}\times C_{s_2}\times\ldots\times C_{s_k}.\]
\end{lemma}

\begin{proof}
The isomorphism is given by the map $\phi:P\rightarrow C_{s_1}\times\ldots\times C_{s_k}$, defined by
\[\phi(p_1^{a_1}p_2^{a_2}\ldots p_k^{a_k})=(a_1,a_2,\ldots,a_k).\]
\end{proof}

These two lemmas allow us to prove Theorems~\ref{HilbertSeriesFactors} and \ref{TraceFunctionFactors}, as shown below.

\begin{proof}[Proof of Theorem \ref{HilbertSeriesFactors}]

Let $P$ be the ranked poset of factors of $n$, ordered by divisibility, and let $n$ have prime factorization $p_1^{s_1}\ldots p_k^{s_k}$, for distinct primes $p_1,\ldots,p_k$.  We have
\begin{eqnarray*}
\mathcal{M}_P(z)&=&\prod_{j=1}^r\mathcal{M}_{C_{s_j}}(z)\\
&=&\prod_{j=1}^r\left(1+s_j(1-z)\right).
\end{eqnarray*}
This gives us
\[H(A(P),z)=\frac{1-z}{1-z\prod_{j=1}^{r}\left(1+s_j(1-z)\right)}.\]
\end{proof}

\begin{proof}[Proof of Theorem \ref{TraceFunctionFactors}]
Let $P$ be the ranked poset of factors of $n$, ordered by divisibility, and let $n$ have prime factorization $p_1^{s_1}\ldots p_k^{s_k}$, for distinct primes $p_1,\ldots,p_k$.  If $\sigma$ permutes primes which appear with the same multiplicity in the factorization of $n$, then $\sigma$ is a factor-shuffling automorphism of $P$.  If $\sigma$ has cycle decomposition $\sigma_1\sigma_2\ldots\sigma_r$, with each cycle $\sigma_j$ of length $i_j$, acting on copies of $C_{s_j}$, then we have
\begin{eqnarray*}
\mathcal{M}_{P^\sigma}(z)&=&\prod_{j=1}^r \mathcal{M}_{\left(\left(C_{s_j}\right)^{i_j}\right)^{\sigma_j}}(z)\\
&=&\prod_{j=1}^r(1+s_j(1-z^{i_j})).
\end{eqnarray*}
This gives us
\[Tr_\sigma(A(P),z)=\frac{1-z}{1-z\prod_{j=1}^r(1+s_j\left(1-z^{i_j})
\right)}\]
\end{proof}

\bibliographystyle{plain}

\bibliography{DirectProducts}

\end{document}